\documentclass[reqno, 11pt]{amsart}
\usepackage{amsthm, amsmath, amssymb, graphicx, tikz,enumerate,paralist,bbm,setspace,bbm}
\usepackage[left=1in,right=1in,top=1in,bottom=1in]{geometry}

\usepackage{xr}
\usepackage[color,final]{showkeys} 
\usepackage[colorlinks=false, pdfstartview=FitV, 
pagebackref=false,hidelinks]{hyperref}

\definecolor{refkey}{gray}{.75}
\definecolor{labelkey}{gray}{.5}

\usetikzlibrary{calc,intersections,through,backgrounds,shapes,patterns}

\newtheorem{theorem}{Theorem}

\newtheorem{lemma}[theorem]{Lemma}

\theoremstyle{definition}
\newtheorem{definition}[theorem]{Definition}

\newtheorem{setup}[theorem]{Setup}
\theoremstyle{remark}
\newtheorem*{remark}{Remark}

\newcommand{\abs}[1]{\left\lvert#1\right\rvert}

\newcommand{\norm}[1]{\left\lVert#1\right\rVert}

\newcommand{\ol}{\overline}

\newcommand{\EE}{\mathbb{E}}
\newcommand{\RR}{\mathbb{R}}

\newcommand{\ZZ}{\mathbb{Z}}

\newcommand{\og}{\bar{g}}

\title{Linear forms from the Gowers uniformity norm}
\author{David Conlon}
\address{Mathematical Institute\\
Oxford OX1 3LB\\
United Kingdom}
\email{david.conlon@maths.ox.ac.uk}

\author{Jacob Fox}
\address{Department of Mathematics\\
MIT\\
Cambridge\\
MA 02139-4307}
\email{fox@math.mit.edu}

\author{Yufei Zhao}
\address{Department of Mathematics\\
MIT\\
Cambridge\\
MA 02139-4307}
\email{yufeiz@math.mit.edu}

\begin{document}

\maketitle

This is a companion note to \cite{CFZrelsz} elaborating on the
concluding remark in \S\ref{sec:concluding} under the heading
\emph{Gowers uniformity norms}.
The purpose of this note is to sketch the argument showing that the relative
Szemer\'edi theorem, Theorem~\ref{thm:rel-sz} in \cite{CFZrelsz}, for $(r+1)$-term arithmetic progressions holds when the linear forms condition on $\nu \colon
\ZZ_N \to \RR_{\geq 0}$ is replaced by an alternate
condition on the Gowers uniformity norm $U^r$:
\begin{equation}\label{eq:Gowers-power}
\norm{\nu - 1}_{U^r}  = o(p^r), \quad \text{where } p :=
1/\norm{\nu}_\infty \leq 1.
\end{equation}
Recall that the Gowers uniformity norm $U^r$ is defined by
\[\norm{f}_{U^r}= \EE\Bigl[ \prod_{\omega \in \{0,1\}^r}
  f(x_0+\omega \cdot {\bf x})\Big\vert x_0,x_1,\ldots,x_r \in \ZZ_N \Bigr]^{1/2^{r}}.
  \]
The application we have in mind is $\nu = p^{-1}1_S$ where $S
\subseteq \ZZ_N$ satisfies $S \subseteq \ZZ_N$ and $p = \abs{S}/N$.

We do not give all the details in this note and we also assume
familiarity with \cite{CFZrelsz}.
 We sketch how to modify the argument in \cite{CFZrelsz} to
show the result under the assumption \eqref{eq:Gowers-power}. As noted
in Footnote~\ref{ft:conditions} on page~\pageref{ft:conditions} of~\cite{CFZrelsz}, the
only hypotheses needed for the proof of the counting lemma are the
strong linear forms condition, as in Lemma~\ref{lem:strong-linear-forms},
and also \eqref{eq:lfc-nu'} in \cite{CFZrelsz}. The Gowers uniformity
hypotheses also implies the conclusion of Lemma~\ref{lem:nu-1-disc}, which gives
the conclusion of Lemma~\ref{lem:nu-upper-regular}, thereby allowing us to apply the
weak regularity lemma, Theorem~\ref{thm:weak-reg-hyp}.

\medskip

As in \cite{CFZrelsz}, we work in the hypergraph setting. Recall that
for a finite set $e$, we write $V_e = \prod_{j \in e} V_j$, where
each $V_j$ is a finite set. We assume this notation for
Definition~\ref{def:Ue} and Lemmas~\ref{lem:GCS} and
\ref{lem:nu-Gowers}.

\begin{definition}[Gowers uniformity norm] \label{def:Ue}
  For any function $g \colon V_e \to \RR$, define
  \[
  \norm{g}_{U^e} := \EE\Bigl[ \prod_{\omega \in \{0,1\}^e}
  g(x_e^{(\omega)})  \Big\vert x_e^{(0)}, x_e^{(1)} \in V_e \Bigr]^{1/2^{\abs{e}}}.
  \]
\end{definition}

There are two notions of Gowers uniformity norm: one for functions $\ZZ_N
\to \RR$ and one for functions $V_e \to \RR$. Observe that the
representation of $\nu \colon \ZZ_N \to \RR_{\geq 0}$ by a
weighted hypergraph $\nu$ in the
proof of the relative Szemer\'edi theorem \cite[\S\ref{sec:relative-szemeredi}]{CFZrelsz} preserves the Gowers
uniformity norm.

The following inequality is the Gowers-Cauchy-Schwarz inequality for
hypergraphs. The proof is by $r$
applications of the standard Cauchy-Schwarz inequality.

\begin{lemma}[Gowers-Cauchy-Schwarz inequality]
  \label{lem:GCS}
  For any collection of functions $g_\omega \colon V_e \to \RR$,
  $\omega \in \{0,1\}^e$, one has
\[
\Bigl\lvert
\EE\Bigl[ \prod_{\omega \in \{0,1\}^e} g_\omega(x_e^{(\omega)})
\Big\vert x_e^{(0)}, x_e^{(1)} \in V_e \Bigr] \Bigr\rvert
\leq \prod_{\omega \in \{0,1\}^e} \norm{g_\omega}_{U^e}.
\]
\end{lemma}

To illustrate the proof of
Lemma~\ref{lem:GCS}, we consider the case $\abs{e} = 2$. We have
\begin{align*}
 & \EE[g_{00}(x,y) g_{01}(x,y') g_{10}(x',y) g_{11}(x',y') \vert x,x'
  \in V_1,\ y,y' \in V_2]^4
  \\
  &=
    \EE[ \EE[g_{00}(x,y) g_{01}(x,y') \vert x \in V_1]
    \EE[g_{10}(x',y) g_{11}(x',y') \vert x' \in V_2] \vert
  y,y' \in V_2]^4
  \\
  &\leq  \EE[ \EE[g_{00}(x,y) g_{01}(x,y') \vert x \in V_1]^2 \vert
  y,y'\in V_2]^2
    \EE[ \EE[g_{10}(x',y) g_{11}(x',y') \vert x' \in V_2]^2 \vert
  y,y' \in V_2]^2
    \\
  &=  \EE[ g_{00}(x,y) g_{00}(x',y) g_{01}(x,y') g_{01}(x',y') \vert x,x' \in V_1,\
  y,y'\in V_2]^2 \\
  &\qquad\qquad
    \EE[ g_{10}(x,y) g_{10}(x',y) g_{11}(x,y') g_{11}(x',y') \vert x,x' \in V_1,\
  y,y' \in V_2]^2
    \\
  &= \EE[ \EE[g_{00}(x,y) g_{00}(x',y) \vert y \in V_2]
  \EE[g_{01}(x,y') g_{01}(x',y') \vert y' \in V_2] \vert x,x' \in V_1]^2 \\
  &\qquad\qquad
    \EE[\EE[g_{10}(x,y) g_{10}(x',y) \vert y \in V_2] \EE[g_{11}(x,y')
    g_{11}(x',y') \vert y'\in V_2] \vert x,x' \in V_1]^2
\\
  &\leq \EE[ \EE[g_{00}(x,y) g_{00}(x',y) \vert y \in V_2]^2 \vert
  x,x' \in V_1]
  \EE[\EE[g_{01}(x,y') g_{01}(x',y') \vert y' \in V_2]^2 \vert x,x' \in V_1] \\
  &\qquad\qquad
    \EE[\EE[g_{10}(x,y) g_{10}(x',y) \vert y \in V_2]^2 \vert x,x'\in V_1] \EE[\EE[g_{11}(x,y')
    g_{11}(x',y') \vert y'\in V_2]^2 \vert x,x' \in V_1]
\\
  &= \EE[ g_{00}(x,y) g_{00}(x',y) g_{00}(x,y') g_{00}(x',y') \vert
  x,x' \in V_1,\ y,y' \in V_2] \\
  &\qquad \EE[ g_{01}(x,y) g_{01}(x',y) g_{01}(x,y') g_{01}(x',y') \vert
  x,x' \in V_1,\ y,y' \in V_2] \\
  &\qquad \EE[ g_{10}(x,y) g_{10}(x',y) g_{10}(x,y') g_{10}(x',y') \vert
  x,x' \in V_1,\ y,y' \in V_2] \\
  &\qquad \EE[ g_{11}(x,y) g_{11}(x',y) g_{11}(x,y') g_{11}(x',y') \vert
  x,x' \in V_1,\ y,y' \in V_2]
  \\
  &= (\norm{g_{00}}_{U^2} \norm{g_{01}}_{U^2}\norm{g_{10}}_{U^2}\norm{g_{11}}_{U^2})^4.
\end{align*}
Both inequalities above are due to the usual Cauchy-Schwarz
inequality. The extension to the general case is straightforward.

\medskip

The following lemma relates the Gowers uniformity norm condition to
certain linear forms within $V_e$.

\begin{lemma}
  \label{lem:nu-Gowers}
  If $\nu_e \colon V_e \to \RR_{\geq 0}$ satisfies $\norm{\nu_e -
    1}_{U^e} = o(1)$, then
  \begin{equation} \label{eq:nu-Gowers}
  \EE\Bigl[ \prod_{\omega \in \{0,1\}^e} \nu_e(x_e^{(\omega)})^{n_{\omega}}
\Big\vert x_e^{(0)}, x_e^{(1)} \in V_e \Bigr] = 1 + o(1)
\end{equation}
for any choices of exponents $n_\omega \in \{0,1\}$.
\end{lemma}

\begin{proof}
  Applying the Gowers-Cauchy-Schwarz inequality, Lemma~\ref{lem:GCS},
  applied with $g_\omega(x_e) = (\nu_e(x_e) - 1)^{n_\omega}$, one gets
    \begin{equation} \label{eq:nu-1-Gowers}
  \EE\Bigl[ \prod_{\omega \in \{0,1\}^e} (\nu_e(x_e^{(\omega)}) - 1)^{n_{\omega}}
\Big\vert x_e^{(0)}, x_e^{(1)} \in V_e \Bigr]
= o(1)
\end{equation}
for any choice of exponents $n_\omega \in \{0,1\}$, as long as they
  are not all zero.
We can write the
  left-hand side of \eqref{eq:nu-Gowers} as
\[
  \EE\Bigl[ \prod_{\omega \in \{0,1\}^e}
  ((\nu_e(x_e^{(\omega)})-1) + 1)^{n_{\omega}}
\Big\vert x_e^{(0)}, x_e^{(1)} \in V_e \Bigr].
\]
The result follows by expanding each parenthesis $((\nu_e(x_e^{(\omega)})^{n_{\omega}}-1) +
1)$ and bounding each term (except for the constant term) using \eqref{eq:nu-1-Gowers}.
\end{proof}

For the rest of this note, we assume the following hypergraph system
setup. Recall that this is the hypergraph system used in the proof of
the relative Szemer\'edi theorem in \cite{CFZrelsz}.

\begin{setup} \label{set:sz-hyp}
  Let $J = \{0,1,2,\dots,r\}$ and $H = \binom{J}{r}$. Write $e_j :=
  J\setminus \{j\} \in H$ for every $j \in J$. Let $V = (J, (V_j)_{j \in J}, r,
H)$ be a hypergraph system. Note that $H$ is the complete $r$-uniform
hypergraph on $r+1$ vertices.
\end{setup}

For a weighted hypergraph $\nu$ on $V$, we write $\norm{\nu}_\infty$ to mean the maximum value taken by any
$\nu_e$, $e \in H$. Throughout we assume that
$\norm{\nu}_\infty \geq 1$.

\medskip

The next two lemmas show that the inputs to the proof of the counting lemma in
\cite{CFZrelsz}  (see
Footnote~\ref{ft:conditions} on page~\pageref{ft:conditions}) remain valid when we assume that
\[
\norm{\nu_e - 1}_{U^e} = o(\norm{\nu}_\infty^{-r}) \quad\text{for all
} e \in H.
\]

\begin{lemma}[Strong linear forms]
  \label{lem:Gowers-strong-linear-forms}
Assume Setup~\ref{set:sz-hyp}. Let
  $\nu$ be a weighted hypergraph on $V$ satisfying
  \[
  \norm{\nu_e - 1}_{U^e} = o(1) \text{ for all } e \in H\setminus\{e_0\}.
  \]
  For each $\iota \in \{0,1\}$ and $e \in H\setminus\{e_0\}$, let
  $g^{(\iota)}_e \colon V_e \to \RR_{\geq 0}$ be a function so that either $g_e^{(\iota)} \leq 1$ or
  $g_e^{(\iota)}\leq \nu_e$ holds. Then
  \begin{multline}
    \label{eq:Gowers-slf}
    \Bigl\lvert \EE\Bigl[
    (\nu_{e_0}(x_{e_0}) - 1)
    \prod_{\iota \in \{0,1\}} \Bigl( \prod_{e \in H \setminus \{e_0\}}
    g_e^{(\iota)}(x_0^{(\iota)}, x_{e\setminus\{0\}})
    \Bigr)
    \Big\vert
    x_0^{(0)},x_0^{(1)} \in V_0, \ x_{e_0} \in V_{e_0}
    \Bigr]\Bigr\rvert \\
    \leq (1 + o(1)) \norm{\nu_{e_0} - 1}_{U^{e_0}} \norm{\nu}_{\infty}^r.
  \end{multline}
\end{lemma}

\begin{proof}
  For each $\iota = 0,1$ and $e \in H\setminus\{e_0\}$, let $\og^{(\iota)}_e$ be either $1$ or $\nu_e$ so that $g^{(\iota)}_e \leq
  \og^{(\iota)}_e$ holds. For $\emptyset \subseteq d \subseteq e_0$, define
  \begin{align*}
    X_d &:= \prod_{\omega \in \{0,1\}^d}(\nu_{e_0}(x_{e_0\setminus d},
    x_d^{(\omega)}) - 1), \\
    Y_d &:= \prod_{\iota \in \{0,1\}} \prod_{\substack{e \in H
        \setminus \{e_0\} \\ e \supseteq d}}
    \prod_{\omega \in \{0,1\}^{d}}
          g_e^{(\iota)}(x_0^{(\iota)}, x_{d}^{(\omega)}, x_{e
            \setminus(d \cup \{0\})}),        \end{align*}
and
\[
        Q_d := \EE \bigl[ X_dY_d \big\vert x^{(0)}_{d \cup
          \{0\}},x^{(1)}_{d \cup \{0\}} \in V_{d \cup \{0\}}, \ x_{e_0 \setminus d} \in V_{e_0 \setminus
    d} \bigr].
\]
We observe that $\abs{Q_\emptyset}$ is equal to the left-hand side of
  \eqref{eq:Gowers-slf} and
  \[
  Q_{e_0} = \EE \Bigl[ \prod_{\omega \in \{0,1\}^{e_0}} (
  \nu_{e_0}(x_{e_0}^{(\omega)}) - 1)
  \Big\vert
  x_J^{(0)},x_J^{(1)} \in V_J
  \Bigr] = \norm{\nu_{e_0} - 1}_{U^{e_0}}^{2^r}.
  \]
  We claim that if $j \in e_0 \setminus d$ then
  \begin{equation}
    \label{eq:Gowers-slf-cs}
    \lvert Q_d \rvert^{1/2^{\abs{d}}}
    \leq (1 + o(1)) Q_{d \cup
      \{j\}}^{1/2^{\abs{d}+1}} \norm{\nu}_\infty,
  \end{equation}
  from which it would follow by induction that
  \[
  \abs{\text{LHS of \eqref{eq:Gowers-slf}}} = \abs{Q_\emptyset} \leq (1 +
  o(1))Q_{e_0}^{1/2^r} \norm{\nu}_\infty^r
  = (1+o(1)) \norm{\nu_{e_0} - 1}_{U^{e_0}} \norm{\nu}_\infty^r
  \]
  as desired.
  Now we prove~\eqref{eq:Gowers-slf-cs}. Let $Y_d = Y_d^{\ni j} Y_d^{\not\ni
    j}$ where $Y_d^{\ni j}$ consists of all the factors in $Y_d$ that
  contain $x_j$ in the argument, and $Y_d^{\not\ni j}$ consists of all
  other factors. Let $\ol Y_d^{\not\ni
    j}$ denote $Y_d^{\not\ni j}$ with all $g^{(\iota)}$ replaced by $\ol
  g^{(\iota)}$. Using the Cauchy-Schwarz inequality and $Y_d^{\not\ni j} \leq \ol Y_d^{\not\ni
    j}$ one has\footnote{The key difference between this argument and the
    proof of Lemma~\ref{lem:strong-linear-forms} in \cite{CFZrelsz} is
    that here we use the Cauchy-Schwarz inequality to bound by $\EE [\EE[ X_d Y_d^{\ni j} \vert x_j \in V_j]^2
    ] \ \EE [(\ol Y_d^{\not\ni j})^2]$, which contains an undesirable
    square $(\ol Y_d^{\not\ni j})^2$, whereas in \cite{CFZrelsz} we
    bound by $\EE [\EE[ X_d Y_d^{\ni j} \vert x_j \in V_j]^2 \ol Y_d^{\not\ni j}
    ] \ \EE [\ol Y_d^{\not\ni j}]$ so that there is no loss in terms of $\norm{\nu}_{\infty}$.}
  \begin{align}
    Q_d^2 &=
    \EE [ \EE[ X_d Y_d^{\ni j} \vert x_j \in V_j]
    Y_d^{\not\ni j} ]^2
\leq
    \EE [\EE[ X_d Y_d^{\ni j} \vert x_j \in V_j]^2
    ] \
    \EE [(Y_d^{\not\ni j})^2 ] \nonumber
    \\
    &\leq
    \EE [\EE[ X_d Y_d^{\ni j} \vert x_j \in V_j]^2
    ] \ \EE [(\ol Y_d^{\not\ni j})^2]
= Q_{d\cup\{j\}} \ \EE [(\ol Y_d^{\not\ni j})^2], \label{eq:Gowers-slf-post-cs}
  \end{align}
  where the outer expectations are taken over all free variables. Note
  that
  \[
  \ol Y_d^{\not\ni j} = \prod_{\iota \in \{0,1\}} \prod_{\omega \in
    \{0,1\}^d}
  \og_{e_j}^{(\iota)}(x_0^{(\iota)},x_d^{(\omega)},x_{e
    \setminus (d\cup\{0\})})
  \]
  is the product of at most $2^{\abs{d}+1}$
  factors of the norm $\nu_{e_j}$. So
  \[
  (\ol Y_d^{\not\ni j})^2 \leq
   \ol Y_d^{\not\ni j} \sup Y_d^{\not\ni j}
   \leq
   \ol Y_d^{\not\ni j} \norm{\nu}_\infty^{2^{\abs{d}+1}}.
  \]
  Since $\norm{\nu_{J\setminus\{j\}} - 1}_{U^{J\setminus\{j\}}} = o(1)$,
  Lemma~\ref{lem:nu-Gowers} implies that $\EE[\ol Y_d^{\not\ni j}]
  = 1 + o(1)$. Thus
  \[
  \EE[(\ol Y_d^{\not\ni j})^2]
  \leq (1+o(1))\norm{\nu}_\infty^{2^{\abs{d}+1}}.
  \]
  So \eqref{eq:Gowers-slf-post-cs} implies \eqref{eq:Gowers-slf-cs},
  as desired.
\end{proof}

\begin{remark}
A straightforward modification of the proof shows that if $g_e \colon
V_e \to \RR_{\geq 0}$ is a function so that $g_e \leq \nu_e$ or $g_e
\leq 1$ for every $e
\in H\setminus\{e_0\}$, then
\[
    \Bigl\lvert \EE\Bigl[
    (\nu_{e_0}(x_{e_0}) - 1)
    \prod_{e \in H \setminus \{e_0\}}
    g_e(x_e)
    \Big\vert
    x_J \in V_J
    \Bigr]\Bigr\rvert \\
    \leq (1 + o(1)) \norm{\nu_{e_0} - 1}_{U^{e_0}} \norm{\nu}_{\infty}^{r/2}.
\]
Indeed, in the proof, the corresponding $\ol Y_d^{\not\ni j}$ now has at
most only $2^{\abs{d}}$ factors, so that $(\ol Y_d^{\not\ni j})^2$ can be
bounded by $\ol Y_d^{\not\ni j} \norm{\nu}_\infty^{2^{\abs{d}}}$, thereby
saving a factor of $2$ in the exponent of $\norm{\nu}_\infty$.
This implies that if $S \subseteq \ZZ_N$, $\nu = p^{-1}1_S$, and $\norm{\nu -
  1}_{U^r} = o(p^{r/2})$ then $S$ contains approximately the correct count
  of $(r+1)$-term arithmetic progressions. This was mentioned in the concluding
  remarks of~\cite{CFZrelsz}.
\end{remark}

\begin{lemma} \label{lem:Gowers-(nu-1)^2}
 Assume Setup~\ref{set:sz-hyp}. Let $\nu$ be a weighted hypergraph on
 $V$ satisfying
 \[
 \norm{\nu_e - 1}_{U^e} = o(\norm{\nu}_\infty^{-r+1}) \text{ for all
 } e \in H.
 \]
 Define $\nu'_{e_0} \colon V_{e_0} \to \RR_{\geq 0}$ by
 \[
 \nu'_{e_0}(x_{e_0}) := \EE\Bigl[ \prod_{e \in H\setminus \{e_0\}}
 \nu_e(x_e) \Big\vert x_0 \in V_0 \Bigr].
 \]
 Then
 \begin{equation} \label{eq:Gowers-(nu-1)^2}
 \EE[(\nu'_{e_0} - 1)^2] = o(1).
\end{equation}
\end{lemma}

Expanding \eqref{eq:Gowers-(nu-1)^2} we see that it suffices to prove
the following lemma.

\begin{lemma}
  \label{lem:Gowers-lf-2}
  Assume Setup~\ref{set:sz-hyp}. Let $\nu$ be a weighted hypergraph on
  $V$ satisfying
  \[
  \norm{\nu_e - 1}_{U^e} = o(\norm{\nu}_\infty^{-r+1}) \text{ for all
  } e \in H.
  \]
  We have
  \[
  \EE\Bigl[ \prod_{e \in H\setminus\{e_0\}} \prod_{\iota \in \{0,1\}}
  \nu_e(x_0^{(\iota)},x_{e \setminus\{0\}})^{n_{e,\iota}}
  \Big\vert
  x_0^{(0)}, x_0^{(1)} \in V_0,\
  x_{J\setminus\{0\}} \in V_{J\setminus\{0\}}
  \Bigr]
  = 1 + o(1)
  \]
  for any choices of exponents $n_{e,\iota} \in \{0,1\}$.
\end{lemma}

\begin{proof}
  [Proof (sketch)]
  It suffices to show, by induction on $\sum_{e,\iota} n_{e,\iota}$,
  that for any $j \in J \setminus\{0\}$,
  \begin{equation} \label{eq:Gowers-lf-2-term}
  \EE\Bigl[(\nu_{e_j}(x_0^{(0)},x_{e_j\setminus\{0\}}) - 1)
  \hspace{-1.5em}\prod_{\substack{e \in H \setminus\{e_0\}, \ \iota \in \{0,1\} \\ (e,\iota) \neq (e_j,0)}}\hspace{-2em}
  \nu_e(x_0^{(\iota)},x_{e \setminus\{0\}})^{n_{e,\iota}}
  \Big\vert
  x_0^{(0)}, x_0^{(1)} \in V_0,\
  x_{e_0} \in V_{e_0}
  \Bigr]
  = o(1).
\end{equation}
We apply the Cauchy-Schwarz inequality to bound~\eqref{eq:Gowers-lf-2-term}, as in the proof of
  Lemma~\ref{lem:Gowers-strong-linear-forms}, doubling (one at a time)
  each vertex in $e_j\setminus\{0\}$. At each application of the
  Cauchy-Schwarz inequality (similar to \eqref{eq:Gowers-slf-post-cs}),
  we obtain a main factor along with a secondary factor that can be upper bounded
  in a way that contributes a factor of $(1+o(1))\norm{\nu}_{\infty}$
  to the bound of~\eqref{eq:Gowers-lf-2-term}. After $r-1$
  applications of the Cauchy-Schwarz inequality, we bound the
  magnitude of \eqref{eq:Gowers-lf-2-term} by
  \[
  (1 + o(1)) \norm{\nu}_{\infty}^{r-1}
    \EE\Bigl[\prod_{\omega \in \{0,1\}^{e_j\setminus\{0\}}}
    (\nu_{e_j}(x_0^{(0)},x_{e_j\setminus\{0\}}^{(\omega)}) - 1) \
    \prod_{\omega \in \{0,1\}^{e_j\setminus\{0\}}}
    \nu_{e_j}(x_0^{(1)},x_{e_j\setminus\{0\}}^{(\omega)})^{n_{e,\iota}}
  \Big\vert
  x_{e_j}^{(0)}, x_{e_j}^{(1)} \in V_{e_j}
  \Bigr]^{1/2^{r-1}}
  \]
  Applying the Cauchy-Schwarz inequality one more time, we can bound
  the second factor by
  \[
    \EE\Bigl[\prod_{\omega \in \{0,1\}^{e_j}}
    (\nu_{e_j}(x_{e_j}^{(\omega)}) - 1)
  \Big\vert
  x_{e_j}^{(0)}, x_{e_j}^{(1)} \in V_{e_j}
  \Bigr]^{1/2^r}
    \EE\Bigl[    \prod_{\omega \in \{0,1\}^{e_j}}
    \nu_{e_j}(x_{e_j}^{(\omega)})^{n_{e,\iota}}
  \Big\vert
  x_{e_j}^{(0)}, x_{e_j}^{(1)} \in V_{e_j}
  \Bigr]^{1/2^r},
  \]
  where the first factor is $\norm{\nu_{e_j} - 1}_{U^{e_j}}$ and the
  second factor is $1 + o(1)$ by Lemma~\ref{lem:nu-Gowers}. It follows
  that the magnitude of \eqref{eq:Gowers-lf-2-term} is bounded by
  $(1+o(1))\norm{\nu}_\infty^{r-1}\norm{\nu_{e_j} - 1}_{U^{e_j}} = o(1)$.
\end{proof}

\end{document}